\documentclass[10pt]{article}
\usepackage[scr=boondox]{mathalfa}
\usepackage{amsmath, amsfonts, amsthm, amssymb, hyperref, setspace, geometry,listings}
\title{\textbf{A new characterization of  $E_8 (p)$ via its vanishing elements}\thanks{\footnotesize  \scriptsize\emph{E-mail address:}
      zsmcau@cau.edu.cn\,(S. Zhang).}}
 
\author{Shengmin Zhang\\
\quad
\\
{\small College of Science,
China Agricultural University,
Beijing 100083, China}}
\date{}
\newtheorem{theorem}{Theorem}[section]

\newtheorem{lemma}[theorem]{Lemma}
\newtheorem{corollary}[theorem]{Corollary}
\theoremstyle{definition}
\newtheorem{definition}[theorem]{Definition}

\newtheorem{question}[theorem]{Question}

\newtheorem{conjecture}[theorem]{Conjecture}
\allowdisplaybreaks

\expandafter\let\expandafter\oldproof\csname\string\proof\endcsname
\let\oldendproof\endproof
\renewenvironment{proof}[1][\proofname]{%
  \oldproof[\bfseries\scshape #1]%
}{\oldendproof}

\usepackage{stackengine,scalerel}
\stackMath

\renewcommand{\leq}{\leqslant}
\renewcommand{\geq}{\geqslant}
\begin{document}
\maketitle
\begin{abstract}
Let $G$ be a finite group, and $g \in G$. Then $g$ is said to be a vanishing element of $G$, if there exists an irreducible character $\chi$ of $G$  such that $\chi (g)=0$. Denote by ${\rm Vo} (G)$ the set of the orders of vanishing elements of $G$. We say a non-abelian group $G$ is V-recognizable, if any group $N$ with ${\rm Vo} (N) = {\rm Vo} (G)$ is isomorphic to $G$. In this paper, we investigate the V-recognizability of $E_8 (p)$, where $p$ is a prime number. As an application, among the 610 primes $p$ with $p<10000$  and $p  \equiv 0,1,4\,(\!\!\!\mod 5)$, we obtain that the method is always valid for confirming the V-recognizability of $E_8 (p)$ for all such $p$ but $  919,1289,1931,3911,4691,5381$ and $7589 $.

\end{abstract}
\section{Introduction}
All groups considered in this paper are finite. Let $G$ be a group, and $g$ be an element of $G$. Then $g$ is called a vanishing element of $G$, if there exists an irreducible character $\chi$ of $G$ such that $\chi (g)=0$. We denote the set of all vanishing elements of $G$ by ${\rm Van} (G)$, and the set of the orders of the vanishing elements in $G$ by ${\rm Vo} (G)$. Let $\pi (G)$ be the set of all prime divisors of $|G|$, and $\pi_e (G)$ be the set of the orders of elements in $G$. Then the {\it vanishing prime graph} of $G$, which is denoted by $\Gamma(G)$, and defined as follows: the vertices of the graph, say $V(\Gamma(G))$, are the prime divisors of elements in ${\rm Vo} (G)$. For any elements $p,q \in V(\Gamma(G))$, $p,q$ are said to be connected, if there exists an element $n \in {\rm Vo} (G)$ such that $pq \,|\,n$. 
  
In 2015, M. F. Ghasemabadi {\it et al.} proposed a conjecture as follows:
\begin{conjecture}[{{\cite{MA}}} or {{\cite[Problem 19.30]{VD}}}]
Let $G$ be a finite group and $H$ be finite non-abelian simple group. Then $G \cong H$ if and only if $|G| = |H|$ and ${\rm Vo}(G) = {\rm Vo}(H)$.
\end{conjecture}
In {{\cite{MA}}}, M. Foroudi Ghasemabadi {\it et al.} obtained that several families of simple groups of Lie type can be uniquely determined by their vanishing elements and orders. Khatami and Babai extends this kind of results to Suzuki's simple groups \cite{MA2}. S. Askary showed in {{\cite{AS}}} that the simple groups of Lie type ${\rm PSL} (3,p)$, where $p$ is a prime, can be uniquely determined by its set of the orders of the vanishing elements and its order. In {{\cite{AS2}}}, S. Askary considered the cases of the simple groups of Lie type $^2 D_{r+1} (2)$ and $^2 D_r (3)$. Since it is possible that some non-abelian groups can be uniquely determined only by their sets of the orders of the vanishing elements without the hypothesis that $|G| = |H|$,  we consider the definition as follows:
\begin{definition}
Let $G$ be a non-abelian group, and $\Omega$ be a set of positive integers. Denote by $v(\Omega)$ the number of isomorphism classes of finite group $N$ such that ${\rm Vo} (N) =\Omega$. It obvious that $v({\rm Vo} (G)) \geq 1$. Then $G$ is called V-recognizable, if $v({\rm Vo} (G)) = 1$.
\end{definition}
In {{\cite{QY}}}, Q. Yan proved that the sporadic simple groups $J_1$, $J_4$ can be uniquely determined by its set of the orders of the vanishing elements, i.e., $J_1$, $J_4$ are V-recognizable. It has also been proved that Suzuki’s simple groups $Sz (2^{2m+1})$, $m \geq 1$ ,${\rm PSL} (2,2^a)$, $a \geq 2$, ${\rm PSL} (2,p)$, $p = \lbrace 5, 7, 8, 17, 211, 269, 283, 293\rbrace$, ${\rm PSL} (3,4)$ and $A_7$ are V-recognizable. In this paper, we continue this work and obtain the following result.
\begin{theorem}\label{main theorem}
Let $r$ be a prime, where $r \equiv 0,1,4\,(\!\!\!\mod 5)$. If for any $s^t < r$, where $s$ is a prime, and $s^t \equiv 0,1,4\,(\!\!\!\mod 5)$, we have $\pi (J_4),\pi (E_8 (s^t)) \not\subseteq \pi (E_8 (r))$, then $E_8 (r)$ is V-recognizable, i.e., for any finite group $G$, $G \cong E_8 (r)$ if and only if ${\rm{Vo}} (G) = {\rm{Vo}} (E_8 (r))$.
\end{theorem}
\begin{corollary}\label{corollary}
Let $p$ be a prime such that  $p \equiv 0,1,4\,(\!\!\!\mod 5)$, and $p < 10000$. If the following hold
$$p \neq 919,1289,1931,3911,4691,5381,7589,$$
then $E_8 (p)$ is V-recognizable. 
\end{corollary}
\begin{question}
Let $r$ be a prime, where $r \equiv 2,3\,(\!\!\!\mod 5)$.
Is the non-abelian simple group $E_8 (r)$ $V$-recognizable?
\end{question}
\section{Preliminary results}
Let $\Omega$ be a finite set of positive integers. $\Pi (\Omega)$ is the simple undirected graph whose vertices are the prime divisors of the elements of $\Omega$. For two vertices of $\Pi (\Omega)$ named $p,q$, they are connected if there exists an element of $\Omega$ divisible by $p$ and $q$. For a finite group $G$, the graph $\Pi (\pi_e (G))$, which is denoted by $GK(G)$, is also said to be the $Gruenberg$-$Kegel$  $graph$ of $G$. The prime graph $\Pi ({\rm Vo}(G))$ is called the vanishing prime graph of $G$ and is denoted by $\Gamma (G)$ in this paper. Let $\mathcal{G}$ be a graph. Then $V (\mathcal{G})$ denotes the vertex set of a graph $\mathcal{G}$, and $n(\mathcal{G})$ indicates the number of connected components of $\mathcal{G}$. Let $G$ be a finite group, $p$ a prime divisor of $|G|$ and $\chi: G \rightarrow \mathbb{C}$ an irreducible complex character of $G$. Then $\chi$ is said to be {\it $p$-defect zero}, if the integer $|G| / \chi(1)$ is not divisible by $p$. 

The next lemma describes the relationship between the solvability of group $G$ and the structure of {\it Gruenberg-Kegel graph} of $G$, especially the connected components.
\begin{lemma}[{{\cite{DO}}}]\label{1}
Let $G$ be a finite group. Then the following statements hold:
\begin{itemize}
\item[(1)] If $G$ is solvable, then $\Gamma (G)$ has at most two connected components;
\item[(2)] If $G$ is non-solvable and $\Gamma (G)$ is disconnected, then $G$ has a unique non-abelian chief factor $S$,
and $n(\Gamma (G)) \leq n(GK(S))$ unless $G$ is isomorphic to $A_7$.
\end{itemize}
\end{lemma}
\begin{lemma}[{{\cite[Proposition 2.1]{DO2}}}]\label{2}
Let $G$ be a non-abelian simple group and $p$ a prime number. If $G$ is of
Lie type, or if $p \geq  5$, then there exists an irreducible complex character $\chi$ of $G$ of $p$-defect zero.
\end{lemma}
In the next lemma, we collect some basic properties relating to the vanishing elements of a group $G$ and the vanishing elements of the quotients of $G$.
\begin{lemma}[{{\cite{DO}}}]\label{3}
Let $N$ be a normal subgroup of a finite group $G$.
\begin{itemize}
\item[(1)] Any character of $G/N$ can be viewed, by inflation, as a character of $G$. In particular, if $xN \in 
{\rm{Van}} (G/N)$, then $xN \subseteq {\rm{Van}} (G)$.
\item[(2)] If $p \in \pi (N)$ and $N$ has an irreducible character of $p$-defect zero, then every element of $N$ of order
divisible by $p$ is a vanishing element of $G$.
\item[(3)] If $m \in {\rm{Vo}}(G/N)$, then there exists an integer $n$ such that $mn \in {\rm{Vo}}(G)$.
\end{itemize}
\end{lemma}
\begin{lemma}[{{\cite{DO}}}]\label{4}
Let $G$ be a group. Assume that $V (\Gamma (G)) \neq \pi (G)$. Then $\Gamma (G)$ is connected.
Moreover, if $G$ is non-solvable, then $G$ has a unique non-abelian composition factor $S$ and $S \cong A_5$.
\end{lemma}
In the following lemma, we investigate the relationship between ${\rm Vo} (G)$ and $\pi_e (G)$ for some non-abelian simple group $G$. Also, several observations of $E_8 (q)$ are listed below as well.
\begin{lemma}\label{6}
Let $G$ be a non-abelian simple group of Lie type, $s$ be a prime and $t$ be an integer larger than $0$. Then the following hold:
\begin{itemize}
\item[(1)] ${\rm{Vo}}(G) = \pi _{e} (G) -\lbrace 1 \rbrace$.
\item[(2)] If $G$ is isomorphic to $E_8 (s^t)$, then $G$ is characterizable.
\item[(3)]  ${\rm{Out}}(E_8 (s^t)) \cong C_t$.
\item[(4)]  $n(GK(E_8 (s^t)))=5$, if $s^t \equiv 0,1,4\,(\!\!\!\mod 5)$.
\end{itemize}
\end{lemma}
\begin{proof}
(1) Since $G$ is a non-abelian simple group of Lie type, we conclude from Lemma \ref{2} and Lemma \ref{3} (2) that ${\rm{Vo}} (G) = \pi_e (G) - \lbrace 1 \rbrace$. 

(2) It follows directly from {{\cite[Main theorem]{KO}}}.

(3) See Table 5 of the third section of chapter 1 in {\cite{CO}}.

(4) See {{\cite[Table 1e]{WI}}}.
\end{proof}

\section{A new characterization of $E_8 (p)$}
Let $G$ be a finite group. Denote by $\mu (G)$ the subset of $\pi_e (G)$ of elements that are maximal under divisibility, and by $\nu(G)$ any subset of $\pi_e (G)$ satisfying the condition $\mu(G) \subseteq \nu(G) \subseteq \pi_e (G)$. Following the notation in {{\cite{AA}}}, we will use $p(\Phi)$ for the maximal power of a prime $p$ lying in the spectrum of a group of Lie type $\Phi$ over a field of characteristic $p$, where the spectrum of a group $G$ is the set of element orders of $G$. Then we have the following lemma:
\begin{lemma}[{{\cite[Theorem]{AA}}}]\label{7}
Let $G$ be a group of Lie type $E_8 (q)$ over a field of characteristic $p$. Suppose that the set $\nu (G)$ is a union of the following sets:
\begin{align*}
(1)~&\lbrace(q+1)(q^2+q+1)(q^5-1), (q-1)(q^2-q+1)(q^5+1), (q+1)(q^2+1)(q^5-1), (q-1)(q^2+1)(q^5+1),\\
&(q + 1)(q^7 -1), (q -1)(q^7 + 1), q^8 -1, (q + 1)(q^3 -1)(q^4 + 1), (q -1)(q^3 + 1)(q^4 + 1), (q^2 + 1)(q^6 -1),\\
&(q^2 - 1)(q^6 + 1), (q^2 - 1)(q^2 + q + 1)(q^4 - q^2 + 1), (q^2 - 1)(q^2 - q + 1)(q^4 - q^2 + 1), (q^2 - 1)(q^6 -q^3 + 1),\\
& (q^2 -1)(q^6 + q^3 + 1), \frac{(q^2+q+1)(q^6+q^3+1)}{(3,q-1)} , \frac{(q^2-q+1)(q^6-q^3+1)}{(3,q+1)} , q^8 + q^7 - q^5 - q^4 - q^3 + q + 1,\\
&q^8 - q^7 + q^5 - q^4 + q^3 - q + 1, q^8 - q^4 + 1, q^8 - q^6 + q^4 - q^2 + 1\rbrace;\\
(2)~& p \cdot \lbrace(q^2 - q + 1)(q^5 + 1), (q^2 + q + 1)(q^5 - 1), (q + 1)(q^6 - q^3 + 1), (q - 1)(q^6 + q^3 + 1), q^7 + 1,q^7 - 1,\\
& (q^3 - 1)(q^4 - q^2 + 1), (q^3 + 1)(q^4 - q^2 + 1), \frac{q^8-1}{
(q-1)(2,q-1)} , \frac{q^8-1}{(q+1)(2,q-1)} , q^6 + 1\rbrace;\\
(3)~& p(A_2)\cdot \lbrace q^6 -1, q^6 +q^3 + 1, q^6 -q^3 + 1,(q^2 +q + 1)(q^4 -q^2 + 1),
(q^2 -q + 1)(q^4 -q^2 + 1),\\
&(q^2 -q + 1)(q^4 - 1),(q^2 + q + 1)(q^4 - 1),(q^2 - 1)(q^4 + 1),(q + 1)(q^5 - 1),(q - 1)(q^5 + 1)\rbrace;\\
(4)~& p(A_3) \cdot \lbrace q^5 - 1, q^5 + 1,(q^4 + 1)(q - 1),(q^4 + 1)(q + 1),(q^3 - 1)(q^2 + 1),(q^3 + 1)(q^2 + 1)\rbrace;\\
(5)~& p(A_4) \cdot\lbrace \frac{q^5-1}{ q-1} , \frac{q^5+1}{ q+1} , q^4 - 1\rbrace;\\
(6)~& p(A_5) \cdot \lbrace
(q^3 - 1)(q + 1),(q^3 + 1)(q - 1), q^4 + 1, \frac{q^4-1}{
(2,q-1) }, q^4 - q^2 + 1\rbrace\\
(7)~& p(D_5) \cdot\lbrace(q^2 + 1)(q - 1),(q^2 + 1)(q + 1), q^3 - 1, q^3 + 1\rbrace;\\
(8)~& p(D_6) \cdot\lbrace 
q^2 + 1\rbrace;\\
(9)~& p(E_6) \cdot \lbrace q^2 - q + 1, q^2 + q + 1, q^2 - 1\rbrace;\\
(10)~& p(E_7) \cdot\lbrace q - 1, q + 1\rbrace;\\
(11)~& \lbrace p(E_8)\rbrace.
\end{align*}
Then $\mu(G) \subseteq \nu(G) \subseteq \pi_e (G)$.
\end{lemma}
\begin{lemma}\label{5}
Let $p < q = s^t$, where $t \in \mathbb{N^{*}}$, $p,s$ are primes with $ p,s^t \equiv 0,1,4\,(\!\!\!\mod 5)$. Then $\pi_e (E_8 (q)) \not\subseteq \pi_e (E_8 (p))$.
\end{lemma}
\begin{proof}
Suppose that the lemma is false, and choose a pair $(p,q=s^t)$ for which it fails. By Lemma \ref{7}, we have:
$$T:=(q^2 + 1)(q^6 -1) = (q^2 +1)(q-1)(q+1)(q^2 +q+1)(q^2 -q+1) \in \pi_e (E_8 (q)) \subseteq \pi_e (E_8 (p)).$$
Since $p<q$, it follows again from Lemma \ref{7} that there exists an element $T'$ in $(3)\sim(11)$ in Lemma \ref{7} such that $T \mid T'$. As we have 
\begin{align*}
&(q^2+1,q-1),(q^2+1,q+1) \leq 2,\,(q^2+1,q^2 -q+1) = (q^2+1,q^2 +q+1) =1,\\
&(q-1,q+1) \leq 2,\,(q-1,q^2 -q+1)=1,\,(q-1,q^2 +q+1) \leq 3,\\
&(q+1,q^2+q+1)=1,\,(q+1,q^2-q+1) \leq 3,\\
&(q^2+q+1,q^2-q+1)=1,
\end{align*}
it is indicated by $p>3$ that the $p$-part of $T$ is at most $q^2+q+1$. Hence the $p'$-part of $T$ is at least $T_0 :=(q^2+1)(q-1)(q+1)(q^2-q+1) = (q^4-1)(q^2-q+1)$. Let $T_0 '$ denote  the $p'$-part of $T'$. Then obviously we have $T_0 \leq T_0 '$. However, it follows from $q \geq p+1$ that:
\begin{align*}
T_0 &=(q^4-1)(q^2-q+1) \geq ((p+1)^4-1)((p+1)^2-(p+1)+1) \\ & =(p^4+4p^3+6p^2+4p)(p^2+p+1)
= p^6+5p^5+11p^4+14p^3+10p^2+4p,
\end{align*}
which is clearly larger than the $p'$-parts of all elements in $(3)\sim(11)$ in Lemma \ref{7}, a contradiction. Thus the result follows.
\end{proof}
\begin{proof}[Proof of Theorem \ref{main theorem}]
If $G$ is isomorphic to $E_8 (r)$, then clearly ${\rm{Vo}}(G) = {\rm{Vo}}(E_8 (r))$. Now we prove the sufficiency. Assume that ${\rm{Vo}}(G) = {\rm{Vo}}(E_8 (r))$. By Lemma \ref{6} (1), it indicates that ${\rm{Vo}}(E_8 (r)) = \pi_e (E_8 (r)) -\lbrace 1 \rbrace$. Hence we get that $ \Gamma(G)= \Gamma(E_8 (r))= GK(E_8 (r))$. By Lemma \ref{6}, we have $n(\Gamma(G)) = n(\Gamma(E_8 (r)))=5$. It follows from Lemma \ref{1} that $G$ is not solvable, and $G$ is disconnected. Since $G$ is not isomorphic to $A_7$, we conclude that $G$ has a unique non-abelian chief factor $S$. Let $N$ be the maximal normal solvable subgroup of $G$, where $M/N$ is a chief factor. As $M/N$ is not solvable, one can easily find that $M/N$ is isomorphic to $S$, i.e., $M/N$ is a non-abelian simple group. Let $A/N :=C_{G/N} (M/N)$, and suppose that $A \neq N$. It follows from the maximality of $N$ that $A/N$ is not solvable. Since $M/N$ is the unique non-abelian chief factor of $G$, we conclude that $M/N \leq A/N$. Hence $M/N$ is abelian, a contradiction. Thus we have $A=N$ and $Z(M/N)=1$. 

Now let $\overline{G} :=G/N$. By N/C Theorem we conclude that $N_{\overline{G}} (\overline{M}) / C_{\overline{G}} (\overline{M})   \lesssim {\rm{Aut}} (\overline{M})$, i.e., $\overline{G} \lesssim {\rm{Aut}} (\overline{M})$. Since $\overline{M} / Z (\overline{M}) \cong {\rm{Inn}} (\overline{M})$, $Z(\overline{M})=1$, and ${\rm{Aut}} (\overline{M}) / {\rm{Inn}} (\overline{M}) \cong {\rm{Out}} (\overline{M})$, it follows that $G/M \cong \overline{G} / \overline{M} \lesssim {\rm{Out}}(\overline{M})$. By Lemma \ref{6} (4), we have $n(\Gamma(G))=5$. Hence by Lemma \ref{1} (2), we get that $n(GK(\overline{M})) \geq 5$.
\begin{itemize}
\item[\textbf{Step 1.}]  $\overline{M} \cong E_8 (r) $.
\end{itemize}
Since $n(GK(\overline{M})) \geq 5$, it follows from {{\cite{KO2}}} and {{\cite{WI}}} that $\overline{M}$ is isomorphic to $E_8 (q)$, where $q =s^t$, and $q \equiv 0,1,4\,(\!\!\!\mod 5)$ or $J_4$. By our hypothesis, there exists $p \in \pi(J_4)$ such that $p \notin \pi (E_8 (r))$. If $\overline{M} \cong J_4$, then we conclude from Lemma \ref{2}  and Lemma \ref{3} (2)(3) that there exists $n \in \mathbb{N^{*}}$ such that $pn \in {\rm Vo} (G) = {\rm Vo}(E_8 (r))$. By Lemma \ref{6} (1), it yields that $p \in \pi (E_8 (r))$, a contradiction. Thus $\overline{M}$ is not isomorphic to $J_4$. Now let $\overline{M} \cong E_8 (q)$, where $q =s^t$, $s$ is a prime, and $q \equiv 0,1,4\,(\!\!\!\mod 5)$. If $q >r$, it follows directly from Lemma \ref{5} that $\pi_e (E_8 (q)) \not\subseteq \pi_e (E_8 (r))$. However, for any $m \in \pi_e (E_8 (q)) = \pi_e (\overline{M})$, Lemma \ref{2} and \ref{3} (2) imply that there exists $n \in \mathbb{N^{*}}$ such that $mn \in {\rm Vo} (G) = {\rm Vo} (E_8 (r))$. By Lemma \ref{6} (2), it yields that $m \in \pi_e (E_8 (r))$. Hence it indicates from the choice of $m$ that $\pi_e (E_8 (q)) \subseteq \pi_e (E_8 (r))$, a contradiction. Hence we conclude that $q \leq r$. If $q <r$, then by our hypothesis, there exists $p \in \pi(E_8 (q))$ such that $p \notin \pi (E_8 (r))$. It follows from Lemma \ref{2}, and Lemma \ref{3} (2)(3) that there exists $n \in \mathbb{N^{*}}$ such that $pn \in {\rm Vo} (G) = {\rm Vo}(E_8 (r))$. By Lemma \ref{6} (1), it yields that $p \in \pi (E_8 (r))$, a contradiction. Thus $\overline{M}$ is not isomorphic to $E_8 (q)$. Finally, we get that $\overline{M}$ is isomorphic to $E_8 (r)$.
\begin{itemize}
\item[\textbf{Step 2.}]  $M \cong E_8 (r) $.
\end{itemize}
Suppose that $N>1$ and let $1 \leq V < N$, where $N/V$ is a chief factor of $G$, and $\widetilde{G}:=G/V$. Since $N$ is solvable, we get that $\widetilde{N}$ is an elementary abelian $q$-group, where $q$ is a prime. It is clear that $q \in \pi(\widetilde{M}) \subseteq \pi(G)$. On the other hand, it follows from Lemma \ref{4} that $\Gamma (G) = \pi (G)$. Hence it implies that $q \in \Gamma (G) \subseteq {\rm{Vo}} (G)$. Since $ \widetilde{M} / \widetilde{N} \cong E_8 (r)$ is a simple group of Lie type, we obtain from  Lemma \ref{2}, and Lemma \ref{3} (2) that any non-trivial element of $\widetilde{M} / \widetilde{N}$ is a vanishing element of $\widetilde{G} / \widetilde{N}$. By Lemma \ref{3} (1), we conclude that $\widetilde{M} \setminus \widetilde{N} \subseteq {\rm{Van}} (\widetilde{G})$, i.e., $\pi_e (\widetilde{M} \setminus \widetilde{N}) \subseteq {\rm{Vo}} (\widetilde{G})$. On the other hand, it yields from Lemma \ref{3} (3) that for any $m \in {\rm{Vo}} (\widetilde{G})$, there exists $n \in \mathbb{N^{*}}$ such that $nm \in {\rm{Vo}} (G)= {\rm{Vo}} (E_8 (r))$. It follows from Lemma \ref{6} (1) that $m \in {\rm{Vo}} (G)={\rm{Vo}} (E_8 (r))$. Hence we have that ${\rm{Vo}} (\widetilde{G}) \subseteq {\rm{Vo}} (G) $. Therefore we get that 
$$\pi_e (\widetilde{M}) = \pi_e (\widetilde{M} \setminus \widetilde{N}) \cup \pi_e (\widetilde{N}) \subseteq {\rm{Vo}} (G) \cup \lbrace 1,q \rbrace ={\rm{Vo}} (G) \cup \lbrace 1 \rbrace = \pi_e (E_8 (r)).$$
As $\pi_e (E_8 (r))  = \pi_e (\overline{M}) \subseteq  \pi_e (\widetilde{M})$, we have $\pi_e (E_8 (r)) = \pi_e (\widetilde{M})$. It follows from Lemma \ref{6} (2) that $E_8 (r) \cong \widetilde{M}$, a contradiction to $V<N$. Thus we have $N=1$ and the result follows.
\begin{itemize}
\item[\textbf{Step 3.}]  $M =G$.
\end{itemize}
It is clear that $G/M \lesssim {\rm{Out}} (M)     =1$. Hence we conclude that $G=M$ and the result follows.
\end{proof}
\begin{proof}[Proof of Corollary \ref{corollary}]
Let $\Omega:=\lbrace 919,1289,1931,3911,4691,5381,7589 \rbrace$, and $r$ be a prime power. By {{\cite{CO}}}, we conclude that:
\begin{align*}
\pi (E_8 (r)) =&  \pi (r(r^8-1)(r^{10}-1)(r^{12}-1)(r^{14}-1)(r^{18}-1))\cup \pi \left(\frac{r^{10} -r^{5} +1}{r^2 -r+1}\right)\cup \pi\left(r^8 -r^4 +1\right)\\
&\cup  \pi \left(\frac{r^{10} +1}{r^2 +1}\right)
\cup \pi\left( \frac{r^{10}+r^5 +1}{r^2 +r +1}\right) \\
=&\pi (r,r^2-1,r^5+1,r^5-1,r^6+1,r^7+1,r^7-1,r^8-r^4+1,\\
&r^{9}+1,r^9-1,r^{10}-r^5+1,r^{10}+r^5+1,r^{10}+1);\\
\pi(J_4) =& \lbrace 2, 3, 5, 7, 11, 23, 29, 31, 37, 43 \rbrace.
\end{align*}
By Theorem \ref{main theorem}, we only need to prove that for any fixed prime $p<10000$, $p \notin \Omega$, we have $\pi (E_8 (r)) \not\subseteq  \pi (E_8 (p))$ for every prime power $r<p$, and $\pi(J_4) \not\subseteq \pi (E_8 (p))$. This can be verified by computational method, and we use the software Python \cite{PY} to achieve this. Corresponding Python code has been listed in Section \ref{appendix} and we are done.
\end{proof}
\section*{Acknowledgement}
The author thanks Prof. Zhencai Shen, whose published work drew the author’s attention to the problem of characterizing simple groups via vanishing elements. The author is immensely grateful to Baoyu Zhang for his insight, patience and selfless guidance throughout the preparation of this paper in the summer of 2023, without whose contributions the present paper would be non-existent. Baoyu Zhang is now at the University of Birmingham; he was a student of Prof. Shen and taught the author advanced group theory. Thanks also extend to Shuangzhi Li for his generous programming support in July 2023 that improves the paper considerably, and it was Baoyu Zhang who requested Li’s help.

\section{Appendix}\label{appendix}
In this section, we present the Python program code used to compute Corollary \ref{corollary} as follows. Copy and run this code in Python, and the result presented in the paper will come out. Due to the restriction of computing power and website response time, the program needs to run for several hours to obtain our results.
\begin{lstlisting}[language=Python]
#Imports and Initialization
import time
import json
from selenium import webdriver
from selenium.webdriver.common.by import By
from selenium.webdriver.common.keys import Keys
from selenium.webdriver.support.ui import WebDriverWait
from selenium.webdriver.support import expected_conditions as EC

browser = webdriver.Chrome()
browser.get("https://zh.numberempire.com/numberfactorizer.php")

#Global Variables and Constants
MAXN = 10000010
SEARCH_RANGE = 10000

is_prime = [1 for i in range(MAXN)]
prime_numbers = []
candidate_r = []
candidate_theta = []
factor_holders = {}
pie = {}
Lambda = [2, 3, 5, 7, 11, 23, 29, 31, 37, 43]

#Loading and Saving Data
def load_pie():
result = {}
with open('pie.json', 'r') as f:
data = json.load(f)
for key in data:
result[int(key)] = data[key]
return result

def save_pie():
with open('pie.json', 'w') as file:
json.dump(pie, file, indent=4)

#Prime Number Generation
def get_prime_number():
   for i in range(2, MAXN):
      if is_prime[i]:
      prime_numbers.append(i)
   for j in range(2 * i, MAXN, i):
      is_prime[j] = 0

#Finding Candidate Primes for r
def get_canditate_r():
   for i in range(len(prime_numbers)):
      if (prime_numbers[i] >= SEARCH_RANGE):
break
      if prime_numbers[i] % 5 == 0 
      or prime_numbers[i] % 5 == 1 
      or prime_numbers[i] % 5 == 4:
      candidate_r.append(prime_numbers[i])

#Exponentiation Function
def qpow(base, exponent):
    result = 1
    while exponent > 0:
        if exponent % 2:
        result *= base
        base *= base
        exponent //= 2
    return result

#Finding Candidate Values for theta
def get_canditate_theta():
    max_exponent = 0
    for i in range(SEARCH_RANGE):
        if qpow(2, i) > candidate_r[len(candidate_r) - 1]:
        max_exponent = i
        break
    unordered_candidate_theta = []
    for i in range(len(prime_numbers)):
        if (prime_numbers[i] >= SEARCH_RANGE):
                break
        for j in range(1, max_exponent):
            theta = qpow(prime_numbers[i], j)
            if (theta >= SEARCH_RANGE):
            break
            if theta % 5 == 0 or theta % 5 == 1 or theta % 5 == 4:
                unordered_candidate_theta.append(theta)
    return sorted(unordered_candidate_theta)

#Calculating $\pi (E_8 (\theta))$ for the prime power theta
def get_factor_holder():
    for i in range(len(candidate_theta)):
        factor_holders[candidate_theta[i]] =[candidate_theta[i], 
        candidate_theta[i] - 1, candidate_theta[i] + 1, 
        qpow(candidate_theta[i], 4) + 1,  qpow(candidate_theta[i], 5) - 1, 
        qpow(candidate_theta[i], 5) + 1, qpow(candidate_theta[i], 6) + 1,
        qpow(candidate_theta[i], 7) - 1, qpow(candidate_theta[i], 7) + 1, 
        qpow(candidate_theta[i], 8) - qpow(candidate_theta[i], 4) + 1,
        qpow(candidate_theta[i], 9) - 1, qpow(candidate_theta[i], 9) + 1, 
        qpow(candidate_theta[i], 10) - qpow(candidate_theta[i], 5) + 1,
        qpow(candidate_theta[i], 10) + 1,
        qpow(candidate_theta[i], 10) + qpow(candidate_theta[i], 5) + 1]
        
def get_factor(number):
    result = []
    now_number = number
    for i in range(len(prime_numbers)):
        if prime_numbers[i] > now_number:
            return set(result)
        if now_number % prime_numbers[i] == 0:
            result.append(prime_numbers[i])
            while now_number  % prime_numbers[i] == 0:
                now_number = now_number / prime_numbers[i]
    return result

def analysis_result(result_str):
    result = []
    split_str = result_str.split("*")
    for str in split_str:
        result.append(int(str.split("^")[0]))
    return result

def query_factor(number):
    search_box = browser.find_element(By.ID, "number")
    search_box.clear()
    search_box.send_keys(str(number))
    submit_button = browser.find_element(By.CSS_SELECTOR,   
                    "input[type='submit'][value='decomposite']")
    submit_button.click()
    time.sleep(1.0)
    result_box = browser.find_element(By.ID, "result1")
    return analysis_result(result_box.get_attribute("textContent"))
def get_pie(number):
    if number in pie:
        return pie[number]
    unordered_result = []
    if number in factor_holders:
        for factor_holder in factor_holders[number]:
            if factor_holder < MAXN * MAXN:
                unordered_result.extend(get_factor(factor_holder))
            else:
                unordered_result.extend(query_factor(factor_holder))
        result = sorted(list(set((unordered_result))))
        return result
    else:
        return result

#Verify if for all prime powers theta < r the union of $\pi(E_8 (\theta)$ 
#and Lambda is not contained in $\pi (E_8 (r))$ 
def check_valid():
    result = []
    for r in candidate_r:
        valid = True
        for theta in candidate_theta:
            if theta < r and set(pie[theta]).issubset(set(pie[r])):
                valid = False
                break
        if valid and not set(Lambda).issubset(set(pie[r])):
            result.append(r)
    return result

#Main Execution Block and Output Result
def solve():
    for i in range(len(candidate_theta)):
        if candidate_theta[i] in pie:
            continue
        pie[candidate_theta[i]] = get_pie(candidate_theta[i])
        save_pie()
    return check_valid()
            
if __name__ == "__main__":
    pie = load_pie()
    get_prime_number()
    get_canditate_r()
    candidate_theta = get_canditate_theta()
    get_factor_holder()
    print(solve())
\end{lstlisting}

\begin{thebibliography}{0}
\bibitem{AS} S. Askary, Recognition by the set of orders of vanishing elements and
order of ${\rm PSL} (3, p)$, {\it An. \c Stiint. Univ. Al. I. Cuza Ia\c si. Mat. (N.S.)} \textbf{118} (2022) 185-193.
\bibitem{AS2} S. Askary, Characterization of some finite simple groups by the set of orders of vanishing elements and order, {\it Ukr. Math. J.} \textbf{73} (2022) 1663-1673.
\bibitem{AA} A. A. Buturlakin, Spectra of groups $E_8 (q)$, {\it Algebra  Logic}, \textbf{57} (2018) 1-8.
\bibitem{CO} J. H. Conway, R. T. Curtis, S. P. Norton, R. A. Parker, R. A. Wilson, Atlas of finite groups, Maximal Subgroups and Ordinary
Characters for Simple Groups, Oxford University Press, New York, 1985.
\bibitem{DO2} S. Dolfi, E. Pacifici , L. Sanus, P. Spiga, On the orders of zeros of irreducible characters, {\it J.
Algebra}, \textbf{321} (2009) 345–352.
\bibitem{DO} S. Dolfi, E. Pacifici, L. Sanus, P. Spiga, On the vanishing prime graph of finite groups, {\it J.
London Math. Soc.} \textbf{82} (2010) 167–183.
\bibitem{MA} M. F. Ghasemabadi, A. Iranmanesh, M. Ahanjideh, A new characterization of some families of finite simple groups, {\it Rend.  Sem. Mat. Univ. Padova}, \textbf{137} (2017) 57–74.
\bibitem{MA2} M. Khatami, A. Babai, Recognition of some families of finite simple groups by order and set of orders of vanishing elements,  {\it Czech. Math. J.} \textbf{68} (2018) 121-130.
\bibitem{KO} A. S. Kondrat’ev, Recognizability by spectrum of groups $E_8(q)$, {\it Trudy Inst. Mat. i Mekh. UrO RAN}, \textbf{16} (2010)  146–149.
\bibitem{KO2} A. S. Kondrat'ev, On prime graph components of finite simple groups, {\it Mat. Sb.}
(1989) \textbf{180} 787-797.
\bibitem{VD} V. D. Mazurov, E. I. Khukhro, eds. {\it Unsolved problems in Group Theory: The Kourovka Notebook}. 20th
edition. Novosibirsk: Russian Academy of Sciences, Siberian Branch, Institute of Mathematics, 2022.
\bibitem{PY} G. Van Rossum, F. L. Drake. Python 3 Reference Manual. Scotts Valley, CA: CreateSpace, 2009.
\bibitem{WI} J. S. Williams, Prime graph components of finite groups, {\it J. Algebra} \textbf{69} (1981) 487-513. 
\bibitem{QY} Q. Yan, The influence of the orders of vanishing elements on the structure of finite groups, {\it Master's thesis of China Agricultural University}, 2019.
\end{thebibliography}
\end{document}